\documentclass[12pt]{amsart}
\usepackage{amsfonts,amssymb,amsmath,amsthm}
\usepackage{url}
\usepackage{enumerate}

\usepackage{graphicx}
\usepackage{tikz}
   \usetikzlibrary{calc}
\usepackage{mathptmx}
\usepackage{helvet}
\usepackage{courier}
\usepackage{txfonts}
\usepackage{tikz} 
\usetikzlibrary{arrows}
\usepackage{type1cm} 
\input xy
\xyoption{all}

\newtheorem{theorem}{Theorem}[section]
\newtheorem{lemma}[theorem]{Lemma}

\newtheorem{corollary}[theorem]{Corollary}
\theoremstyle{definition}
\newtheorem{definition}[theorem]{Definition}
\newtheorem{remark}[theorem]{Remark}
\newtheorem{example}[theorem]{Example}
\numberwithin{equation}{section}

\newcommand{\sRips}{\mathrm{sRips}}
\newcommand{\Rips}{\mathrm{Rips}}
\newcommand{\Diam}{\mathrm{Diam}}
\newcommand{\Nerve}{\mathrm{Nerve}}

\newcommand{\NN}{\mathbb{N}}
\newcommand{\RR}{\mathbb{R}}
\DeclareMathOperator*{\arginf}{arg\,inf}

\newcommand\K{7}

\author{Bo\v stjan Leme\v z}
\address{{Faculty of Mathematics and Physics, University of Ljubljana, Jadranska 21, SI-1000 Ljubljana, Slovenia}}
\email{bostjan.lemez@fmf.uni-lj.si}

\author{\v Ziga Virk}
\address{Institute IMFM, Jadranska 19, SI-1000 Ljubljana; and Faculty of Computer and Information Science, University of Ljubljana, Ve\v cna pot 113, SI-1000 Ljubljana, Slovenia}
\email{ziga.virk@fri.uni-lj.si}
\thanks{Research was supported by Slovenian Research Agency grants No. N1-0114, J1-4001, J1-4031, and P1-0292.}

\begin{document}

\title[Finite reconstruction with selective Rips complexes]{Finite reconstruction with selective Rips complexes}

\begin{abstract}
Selective Rips complexes corresponding to a sequence of parameters are a generalization of Rips complexes utilizing the idea of thin simplices. In this paper we prove they can be used to reconstruct the homotopy type of a closed Riemannian manifold $X$ using a finite sample $Y$ of $X$.
In particular, for any sequence of parameters with positive limit  and any closed Riemannian manifold $X$ we prove the following: there exists a proximity parameter $\delta$, such that for each metric space $Y$ that is at Gromov-Hausdorff distance less than $\delta$ to $X$, the selective Rips complex of $Y$ attains the homotopy type of $X$. 
This result is a generalization of Latchev's reconstruction result from Rips complexes to selective Rips complexes. When restricted to Rips complexes, our approach yields a novel proof for the Latschev's theorem. We also present a functorial setting, which is new even in the case of Rips complexes. The latter provides an interval of constant persistent homology, where homology is isomorphic to that of $X$. 
\end{abstract}
\maketitle

\section{Introduction}

Rips complexes are one of the most widespread combinatorial constructions arising from metric spaces. Originally appearing in \cite{Viet}, they have since been used in geometric group theory \cite{Gromov}, coarse geometry \cite{Comb}, and, above all, in computational topology \cite{EH} in the context of persistent homology. Their combinatorial simplicity results in significant algorithmic advantages \cite{bauer2019ripser}. As a result they are perhaps the most popular choice of complexes (and filtrations) in topological data analysis. However, understanding how they encode the geometry of the underlying metric space $X$ is challenging. This task is typically treated in the following questions:
\begin{enumerate}
 \item Reconstruction result: Does the Rips complex at small scale attain the homotopy type of $X$?
  \item Finite reconstruction result: Does the Rips complex of a finite set similar to $X$ attain the homotopy type of $X$ at appropriate scales?
  \item How does homology of Rips complexes encode other geometric information about $X$?
\end{enumerate}

The reconstruction result for Rips complexes was first proved for closed Riemannian manifolds  (i.e. compact Riemannian manifolds without boundary) by Hausmann \cite{Hausmann}. The first finite reconstruction result for the same class of spaces was then proved by Latschev \cite{Latschev} using the Gromov-Haus\-dorff metric as a measure of proximity. Further reconstruction results, have later been obtained for Rips complexes \cite{Att}, Čech complexes \cite{Chazal, W}, Delauney complexes \cite{Dey} and witness complexes \cite{Guibas}. In many of these cases the proof used the nerve theorem. Question (3) turns out to be quite challenging. Its various aspects have recently been studied in \cite{AA, Ad5, ACos, Feng, Memoli, Saleh, Shukla, Virk1-dim, VirkSRips1, ZV2, ZVCounterex}. In \cite{ZV2} the author proved that many simple closed geodesics in a geodesic metric space can be detected using persistent homology in dimensions one, two and three. In order to significantly expand the collection of simple closed geodesics detectable by persistent homology, the author introduced selective Rips complexes in \cite{VirkSRips1}. The underlying geometric idea is to control the ``thickness'' of simplices by restricting the Rips complex to thin simplices. The restriction is controlled through various parameters, see Definition \ref{sRips} for details. Further work on selective Rips complexes includes the reconstruction result \cite{LemezVirk}, the ability to detect all locally isolated minima of the distance function on a metric space \cite{Gor}, and the ability to detect many simple closed geodesics in hyperbolic surfaces \cite{Jelenc}. These relate to questions (1) and (3) above. The \textbf{main result} of this paper is a finite reconstruction result for selective Rips complexes in Theorem \ref{Latschev}, i.e., a positive answer to question (2) in the grand scheme of the above three questions. When restricted to Rips complexes, our approach provides a novel proof of the initial result of Latschev. An additional benefit of our approach is that the homotopy equivalence is functorial, i.e., that there exists a region of parameters at which the natural inclusions of selective Rips complexes are homotopy equivalences. This is a new result even in the case of Rips complexes and provides additional information on persistent homology. Functoriality of reconstruction result (1) for Rips complexes was proved in \cite{VirkRipsAsNerves} and used in \cite{ZVCounterex}.

The \textbf{main idea of the proof} is to construct isomorphic good covers $\mathcal{C}$ of $X$ and $\mathcal{W}$ of a selective Rips complex of $Y$ (where $Y$ is an approximation of $X$), and then use a functorial nerve theorem of \cite{VirkRipsAsNerves} to obtain our main result. The structure of the paper is the following.  Preliminaries on selective Rips complexes are given in Section 2. Geometric lemmas adapting some of the technical steps of \cite{Latschev} are provided in Section 3. These essentially describe geometric conditions under which the selective Rips complexes of finite approximations of discs or their star-shaped subsets are contractible. In Section 4 we carefully construct two covers. Section 5 carries out the main argument as described above.

Selective Rips complexes have a potential to act as a finer and easy computable version of Rips complexes. We expect them encode more geometric information and allow us to control the level of details extracted by persistent homology. It is reasonable to expect such control to be beneficial in theoretical and practical applications.

\section{Preliminaries} 

We start by defining the involved construction of simplicial complexes. 
Given a metric space $(X,d)$ and a scale $r>0$, the \textbf{Rips} complex (sometimes also referred to as Vietoris-Rips complex) $\Rips(X,r)$ is an abstract simplicial complex with vertex set $X$, defined by the following rule: a finite $\sigma\subseteq X$ is a simplex if $\Diam(\sigma)<r$. In this paper we focus on selective Rips complexes. They were first introduced in \cite{VirkSRips1}. The following more general definition first appears in \cite{LemezVirk}.

\begin{definition}\label{sRips}
Let $(X,d)$ be a metric space and let $r_1\geq r_2\geq \ldots$ be positive scales forming a sequence $\tilde r=(r_1,r_2, \ldots)$.  \textbf{Selective Rips complex} $\sRips(X; r_1,r_2,\ldots)=\sRips(X;\tilde r)$ is an abstract simplicial complex defined by the following rule: a finite subset $\sigma\subseteq X$ is a simplex if for each positive integer $i$, the set $\sigma$ can be expressed as a union of $i$-many sets of diameter less than $r_i$.
\end{definition}

\begin{figure}[ht!]
	 	\begin{center}
	    \includegraphics[width=23em]{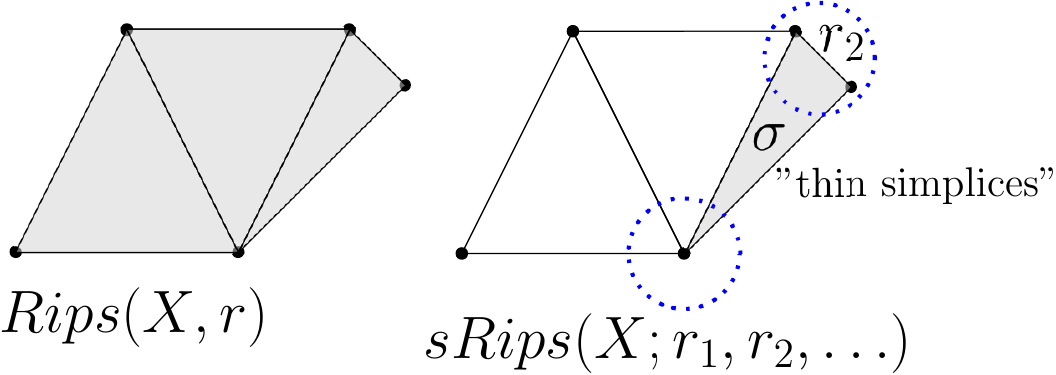}
	    	\caption{The planar simplicial complex on the left is the Rips complex of a five point set $X$ at scale $r$, where $r>0$ is slightly more than the length of the included edges, and less than the distance between the points not connected by an edge. On the right we see a selective Rips complex with $r_1=r$ and $r_2$ being slightly larger than the length of the short edge indicated on the figure. We see it contains  only the ``$r_2$-thin'' triangle amongst all the triangles of the Rips complex on the left. As $r_1=r$, the edges are the same as the edges of $\Rips(X,r)$.}
	    \label{figsRips}
	  \end{center}
		\end{figure}

Note that $\sRips(X; r_1,r_2,\ldots)$ is a subcomplex of $\Rips(X,r_1)$ and that \\  $\sRips(X; r,r,r,\ldots)=\Rips(X,r)$. By increasing the scales $r_i$ in any way that respects the above monotonicity of $r_i$ in index $i$, we form a nested sequence of simplicial complexes called a filtration. Scale $r_i$ denotes the ``width'' of simplices in dimension $i$ in the following sense: for $n > i$ the vertices $v_1, v_2, \ldots, v_n $ of an $(n-1)$-simplex  $\sigma$ in $\sRips(X; r_1,r_2,\ldots)$ can be partitioned (clustered) in $i$ sets of diameter less than $r_i$. These clusters can be thought of as vertices of an $(i-1)$ simplex ``thickened'' by $r_i$. For an example see Figure \ref{figsRips}.
We will also use the following notation. For $\tilde{r}=(r_1,r_2,...)$, where $r_1\geq r_2\geq \ldots>0$, let $\displaystyle{r_\infty=\lim_{n\to \infty}r_n}$. Moreover, we write $\alpha>\tilde{r}$ if $\alpha>r_1$ and $\tilde{r}>\beta$ if $ r_\infty>\beta$.

We now define combinatorial and topological prerequisites for the use of a nerve theorem.
Simplicial maps $f,g: K_1 \to K_2$ between simplicial complexes are \textbf{contiguous} if for each $\sigma\in K_1$, $f(\sigma)\cup g(\sigma)$ is  a simplex in $K_2$. Contiguous maps are homotopic (see \cite[p. 130]{Spanier}).
A cover $\mathcal{U}$ of a metric space is \textbf{good} if each finite intersection of elements from $\mathcal{U}$ is either empty or contractible. The \textbf{nerve} of $\mathcal{U}$ is the simplicial complex $\Nerve(\mathcal{U})$ defined by the following declarations:
\begin{itemize}
 \item Vertices are elements of $\mathcal{U}$.
 \item $\sigma$ is a simplex if and only if $\displaystyle\bigcap_{U\in \sigma} U \neq \emptyset$.
\end{itemize}

We will be using the Functorial Nerve Theorem as presented in \cite{VirkRipsAsNerves}. For details and further discussion see \cite{VirkRipsAsNerves}.

\begin{theorem}[Functorial nerve theorem, an adaptation of Lemma 5.1 of \cite{VirkRipsAsNerves}]
\label{ThmNerve}
 Suppose $\mathcal{U}$ is a good open cover of a metric space $X$. Then $X \simeq \Nerve(\mathcal{U})$. 
 
 If $\mathcal{V}$ is another good open cover of $Y \subset X$ subordinated to $\mathcal{U}$ (i.e., if $\forall V \in \mathcal{V} \ \exists U_V\in \mathcal{U}: V \subseteq U_V$), then the diagram 
 $$
 \xymatrix{
 X\ar[r]^{\simeq \qquad} & \Nerve(\mathcal{U})\\
 Y  \ar@{^(->}[u] \ar[r]^{\simeq \qquad}& \Nerve(\mathcal{V}) \ar[u]
 }
 $$
 commutes up to homotopy, with $\Nerve(\mathcal{V})\to \Nerve(\mathcal{U})$ being the simplicial map mapping $V \mapsto U_V$ and the horizontal homotopy equivalences arising from partitions of unity corresponding to the involved covers.
\end{theorem}

We next define geometric properties of treated spaces.
Let $(X,d)$ be a metric space and $q>0$. The open $r$-neighborhood of $Z \subset X$ is  denoted by $N(Z,q) = \{y \in X \mid d(x,y)<q, \ \forall x \in Z \}$. The open $r$-ball of $x \in X$ is  denoted by $B(x,q) = \{y \in X \mid d(x,y)< q \}$. 
A path between points $x,y\in X$ is called a \textbf{geodesic} from $x$ to $y$, if its length is $d(x,y)$.

Let $(X,d)$ be a Riemannian manifold. Then there exists $r(X)\geq 0$ as the least upper bound of the set of real numbers $r$ satisfying the following (see \cite{Hausmann} for details): 
\begin{enumerate}
\item For all $x,y\in X$ such that $d(x,y)<2r$ there exists a unique geodesic joining $x$ to $y$ of length $2r$.
\item Let $x,y,z,u \in X$ with $d(x,y)<r$, $d(u,x)<r$, $d(u,y)<r$ and $z$ be a point on the shortest geodesic joining $x$ to $y$. \\Then $d(u,z)\leq \max\{d(u,x),d(u,y)\}$.
\item If $\gamma$ and $\gamma'$ are arc-length parametrized geodesic such that $\gamma(0)=\gamma'(0)$ and if $0\leq s$, $s'<r$ and $0\leq t<1$, then $d(\gamma(ts),\gamma'(ts'))\leq\\ d(\gamma(s),\gamma'(s'))$.
\end{enumerate}

A compact Riemannian manifold $X$ has $r(X)>0$. We will denote $r(X)$ by $\rho$, and we call it a \textbf{star radius}. We say that $A\subseteq X$ is a \textbf{star-shaped} with center at $x_0\in A$ if for each $x\in A$, all geodesic from $x$ to $x_0$ are contained in $A$.

It remains to formalize the proximity of metric spaces.
The \textbf{Hausdorff distance} between subsets $Y,Z\subset X$ of a metric space $X$ is defined as 
$$
d_{H}(Y,Z) = \arginf_{s>0} \left\{ Z \subseteq N(Y,s) \textrm{ and } Y \subseteq N(Z,s)  \right\}.
$$

The \textbf{Gromov-Hausdorff} distance between metric spaces $Y$ and $Z$ is defined as 
\begin{multline*}
d_{GH}(Y,Z) =\\ \inf_{U,i,j} \left\{ d_H(i(Y),j(Z)) \mid i\colon Y \hookrightarrow U \textrm{ and } j\colon Z \hookrightarrow U \textrm{ are isometric embeddings }  \right\}.
\end{multline*}

It is well known that these two distances are in fact metrics when restricted to compact metric subspaces or isometry classes of compact metric spaces respectively. 
Observe that for $Y,Z\subset X$ we have $d_H(Y,Z)\leq d_{GH}(Y,Z)$.

The main result of the paper is a finite reconstruction result for selective Rips complexes, i.e., Theorem \ref{Latschev}. 

\begin{theorem}\label{Latschev}
Let $X$ be a closed Riemannian manifold. Then there exists a positive number $\varepsilon_0$, such that for any $M\in \{2,3,\ldots\}$ and any sequence $\tilde{r}$ with $\varepsilon_0 > \tilde r> \varepsilon_0 /M$, there exists $\delta>0$ such that for every metric space $Y$ with  $d_{GH}(X,Y) < \delta$ we have $\sRips\left (Y; \tilde{r}\right )\simeq X$.

If $\tilde{s}=(s_1,s_2,...)$ is another sequence of scales with $\varepsilon_0 > \tilde s> \varepsilon_0 /M$ and $s_i \leq r_i, \forall i$, then the natural inclusion $\sRips\left (Y; \tilde{s}\right ) \to \sRips\left (Y; \tilde{r}\right )$ is a homotopy equivalence.
\end{theorem}

Proximity parameter $\delta$ depends on $X$, $\varepsilon_0$, and $M$. The role of $M$ is twofold. In the first part of Theorem \ref{Latschev} it merely ensures that the scales $r_i$ have a positive limit. In the second part it formalizes a common lower bound on the limit of the scales. As a special case we obtain the following generalization of \cite{Latschev} to a functorial setting.

\begin{corollary}\label{LatschevCor}
Let $X$ be a closed Riemannian manifold. Then there exists a positive number $\varepsilon_0$, such that for any $L \in (0, \varepsilon_0)$, there exists $\delta>0$ such that for every metric space $Y$ with  $d_{GH}(X,Y) < \delta$ we have $\Rips\left (Y, r\right )\simeq X, \ \forall r \in (L,\varepsilon_0)$. Furthermore, if $s \in (L,r)$, then the natural inclusion $\Rips\left (Y, s\right ) \to \Rips\left (Y, r\right )$ is a homotopy equivalence.
\end{corollary}

The two results thus imply that for each $X$ there exists an (open) interval of scales on which persistent homology of $\sRips\left (Y; \tilde{r}\right )$ is constant and isomorphic to $X$. The right endpoint of the said interval is $\varepsilon_0$, the left endpoint $L$ can be chosen arbitrarily close to $0$, and with these choices the required proximity $\delta$ can be proved to exist. 

\section{$\tilde r$-crushings}

In this section we explain how geometry of metric spaces close to an Euclidean disc or its star-shaped subsets results in contractible selective Rips complexes of the said spaces. 
We first define a discrete version of crushings \cite{Hausmann} on metric spaces. A similar version in the context of Rips complexes has also been used in \cite{Latschev}.

\begin{definition}\label{definitionCrushing1}
Let $X$ be a metric space and let $\tilde{r}\geq r_\infty>0$ be a sequence of positive numbers. Suppose we are given a non-empty subset $A\subsetneq X$ and a point $b\in X\setminus A$ satisfying:
$$B(a,r_i)\subset B(b,r_i)\textrm{ for all }a\in A\textrm{ and for all positive integers }i.$$
Then the map $c:X\to X\setminus A$ mapping $A$ to the point $b\in X$ and leaving all other points fixed is called an $\tilde{r}$-\textbf{crushing}. If $A$ is a singleton, map $c$ is called an elementary $\tilde{r}$-crushing.
\end{definition}

As a demonstrative example consider $\{0, 1, 2, \ldots, 10 \} \subset \RR$ as a subspace of the Euclidean line. If $r_i > 2, \forall i$, then $\{0, 1, 2, \ldots, 10 \} \to \{ 1, 2, \ldots, 10 \}$ is an $\tilde{r}$-crushing (with $a=0$ and $b=1$ in the notation of Definition \ref{definitionCrushing1}). In fact, there exists a sequence of $\tilde{r}$-crushings from $\{0, 1, 2, \ldots, 10 \}$ to the singleton $\{10\}$ (or any other singleton of the space).

Observe that an $\tilde{r}$-crushing consists of a sequence of elementary $\tilde{r}$-crush\-ings. The following lemma states that $\tilde{r}$-crushings induce contiguous maps (and thus homotopy equivalences) on selective Rips complexes.

\begin{lemma}
Map $c$ induces a simplicial map  $\sRips\left (X;\tilde{r} \right ) \to \sRips\left (X\setminus A;\tilde{r}\right ) \subseteq \sRips\left (X;\tilde{r} \right )$ which is contiguous to the identity on $\sRips (X; \tilde{r} )$.
\end{lemma}

\begin{proof}
We show the statement for $A=\{a\}$, a general case follows by a sequence of elementary $\tilde{r}$-crushings. 
 We see that $\Diam(\{a,b\})<r_i, \forall i$. Let $\sigma$ be a simplex in $\sRips\left (X;\tilde{r}\right )$ such that $a\in  \sigma$. Then  $(\sigma\setminus \{a\})\cup \{b\}$ is a simplex in $\sRips\left (X\setminus \{a\};\tilde{r}\right )$ and $\sigma\cup \{b\}$ is a simplex in $\sRips\left (X;\tilde{r} \right )$. 
\end{proof}


A metric space is $\tilde{r}$-\textbf{crushable} if there is a sequence of $\tilde{r}$-crushings to a space of diameter less than $\tilde{r}$. Clearly, the selective Rips complex with parameters $\tilde r$ of such a space is contractible in the usual sense. On a related note, if $x\in X$ with $B(x,r_\infty)=X$, then $X \to \{x\}$ is an $\tilde r$-crushing and thus $\sRips(X; \tilde r)$ is contractible.

\begin{corollary}\label{CorollaryCrushable}
If $X$ is $\tilde{r}$-crushable, then $\sRips (X; \tilde{r} )$ is contractible.
\end{corollary}

\begin{remark}
Let $X$ be a metric space and $x\in X$. If positive radii $r_i$ converge to $0$, the conditions of Definition \ref{definitionCrushing1} can not be satisfied.
\end{remark}

The following is a counterexample to the conclusion of Theorem \ref{Latschev}, in the case the sequence of parameters $r_i$ converges to $0$. Given $\delta>0$ a subset $F \subset X$ of a metric space $X$ is $\delta$-\textbf{dense} if $\forall x \in X \ \exists a\in F: d(a,x)<\delta.$

\begin{example}
\label{Example1}
Let $D=[0,1]$ be the unit interval. Let an $\varepsilon_0 \in (0,1)$  and let $\tilde{r}$ be a sequence of positive radii which converges to $0$ such that the following holds:
\begin{itemize}
\item $\varepsilon_0\geq r_1$,
\item $r_{i-1}>(i+1)r_i$ for all positive integer $i$.
\end{itemize}

For each $\delta>0$ let $F_\delta$ be a finite, $\delta$-dense subset of $D$:
$$D=\bigcup_{x\in F_\delta}B(x,\delta).$$

Let us fix a set $F_\delta$. There exists a positive integer $n$ (for any $\varepsilon_0$) such that the pairwise distance between points in $F_\delta$ is more than $2r_{n-1}$. Then there is an $x\in D$ such that $F_\delta\cap \overline{B(x,r_{n-1})}=\emptyset$.

Let $Y=\{x_0,x_1,\ldots,x_n\}\subseteq B\left(x,\frac{r_{n-1}}{2}\right )$ be a set of $n+1$ points in $D$ (they can be evenly spaced) such that the pairwise distances are larger than $r_n$. Such a set exists because $r_{n-1}>(n+1)r_n$.  Note that $F_\delta\cup Y$ is also finite, $\delta$-dense subset of $D$. Since diameter of any subset of $Y$ (and also of $F_\delta\cup Y$) with at least two points is greater than $r_n$ it is clear that $\sRips (F_\delta\cup Y; \tilde{r} )$ has no $n$-dimensional simplices. On the other hand, each subset $Y\setminus\{x_i\}$ is an $(n-1)$-simplex in $\sRips (F_\delta\cup Y; \tilde{r} )$. Therefore, $Y$ forms a non trivial $n$-dimensional homotopy class.

The same argument also applies to a closed manifold $S^1$.

%
\end{example}

The following is an adaptation of \cite[Lemma 2.1.]{Latschev}.

\begin{lemma}\label{lemmaLatschev}
Let $B(0,\alpha)\subset \RR^n$ be the standard open ball of radius $\alpha > 0$, let $D\subseteq B(0,\alpha)$ be a star-shaped set centered at $0$, and let $\tilde{r}$ with $\alpha \geq\tilde{r}\geq r_\infty>0$ be fixed. Then there exists $\delta_1>0$ such that every finite $\delta_1$-dense subset $F\subset D$ is $\tilde{r}$-crushable.
\end{lemma}

\begin{proof} Let $F$ be a finite subset of $D$ whose density as a function of $\tilde r$ and $\alpha$ we will prescribe later. Let $y\in F$ be a point with maximal distance $r'=d(y,0)$ to the origin. For each $i\in \{\infty, 1,2,\ldots\}$ we define 
$$
I_{y,i}=\bigcap_{x\in \overline{B(y,r_i)}\cap\overline{B(0,r')}}B(x,r_i).
$$

Set $\displaystyle \bigcap_{i\in \NN}I_{y,i}$ consists of some (in the case when $D=B(0,\alpha)$, it consists of all) of the the admissible points in $B(0,r')$ that facilitate an $\tilde r$-crushing $F \to F \setminus \{y\}$. We aim to prove this set contains a ball of radius $\delta_1$, which is independent of $y$.

Sets $\overline{B(y,r_i)}\cap\overline{B(0,r')}$ and $I_{y,i}$ are rotationally symmetric hence it suffices to treat their two-dimensional slices containing $0$ and $y$. Choose $p_i, p'_i \in B(0,\alpha)$ with $d(p_i,y)=d(p'_i,y)=r_i$ and $d(p_i,0)=d(p'_i,0)=r'$ as in Figure \ref{figLat1} on the left. Set $\overline{B(y,r_i)}\cap\overline{B(0,r')}$ is the light-grey region, while elementary geometric observations imply $I_{y,i}=B(p_i,r_i)\cap B(p'_i,r_i)$ is the dark-grey region on the mentioned figure.

We first claim that for $k\geq j$ we have $I_{y,k}\subseteq I_{y,j}$. Choose any $z\in I_{y,k}$ and $x_j\in \overline{B(y,r_j)}\cap\overline{B(0,r')}$. If $d(z,x_j)\leq r_k$ define $x_k=x_j$, else define $x_k$ as the point on the line segment from $y$ to $x_j$ with $d(y,x_k)=r_k$. Note that $x_k \in \overline{B(y,r_k)}\cap\overline{B(0,r')}$ and thus
$$
d(z,x_j)\leq d(z, x_k) + d(x_k,x_j) \leq r_k + (r_j-r_k)=r_j.
$$
Thus $z\in I_{y,j}$ and the claim is proved, see Figure \ref{figLat1} on the right.

 \begin{figure}[ht!]
	 	\begin{center}
	    \includegraphics[width=17em]{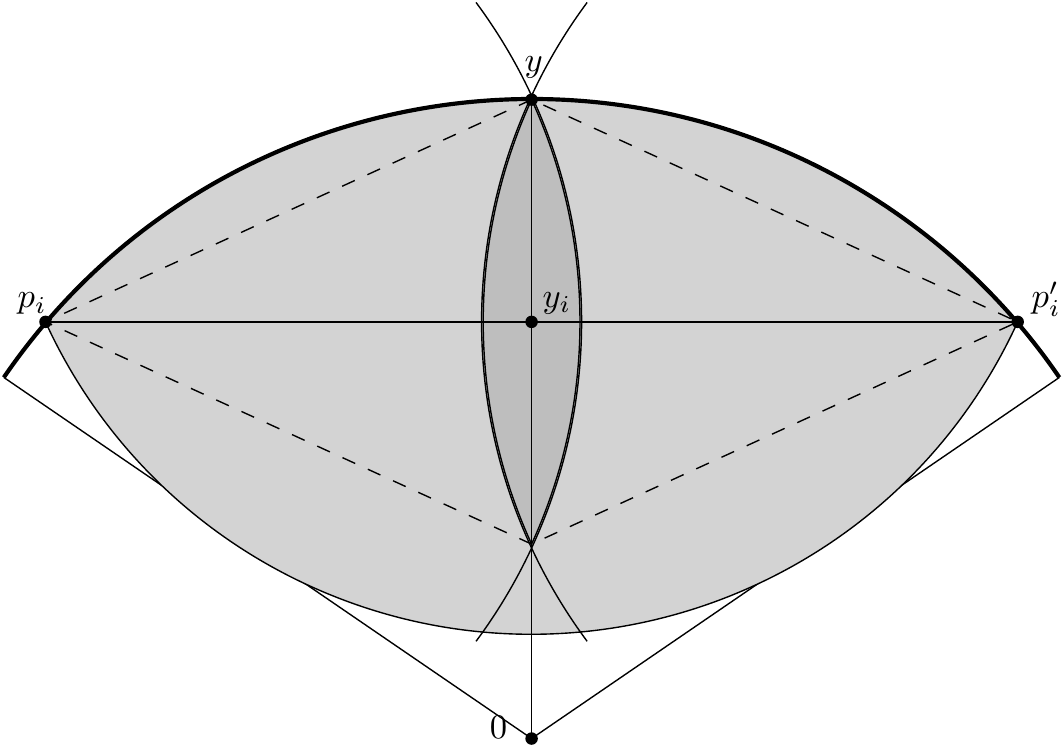}
	    \includegraphics[width=17em]{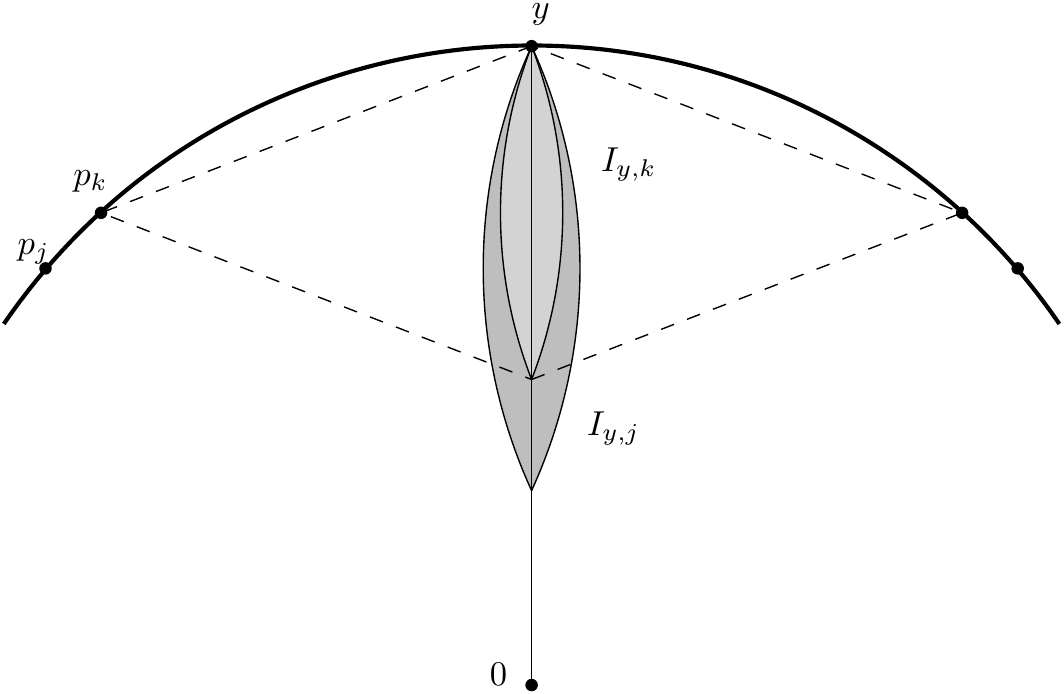}
	    	\caption{A detail from the proof of Lemma \ref{lemmaLatschev}.}
	    \label{figLat1}
	  \end{center}
		\end{figure}

 \begin{figure}[ht!]
	 	\begin{center}
	    \includegraphics[width=17em]{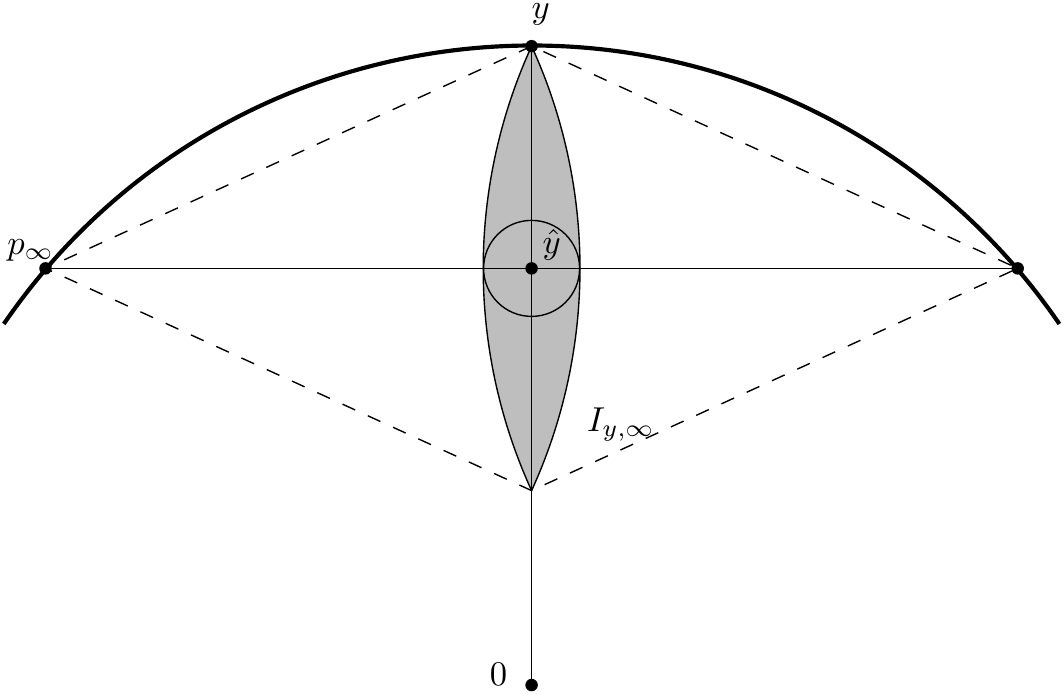}
	    \includegraphics[width=17em]{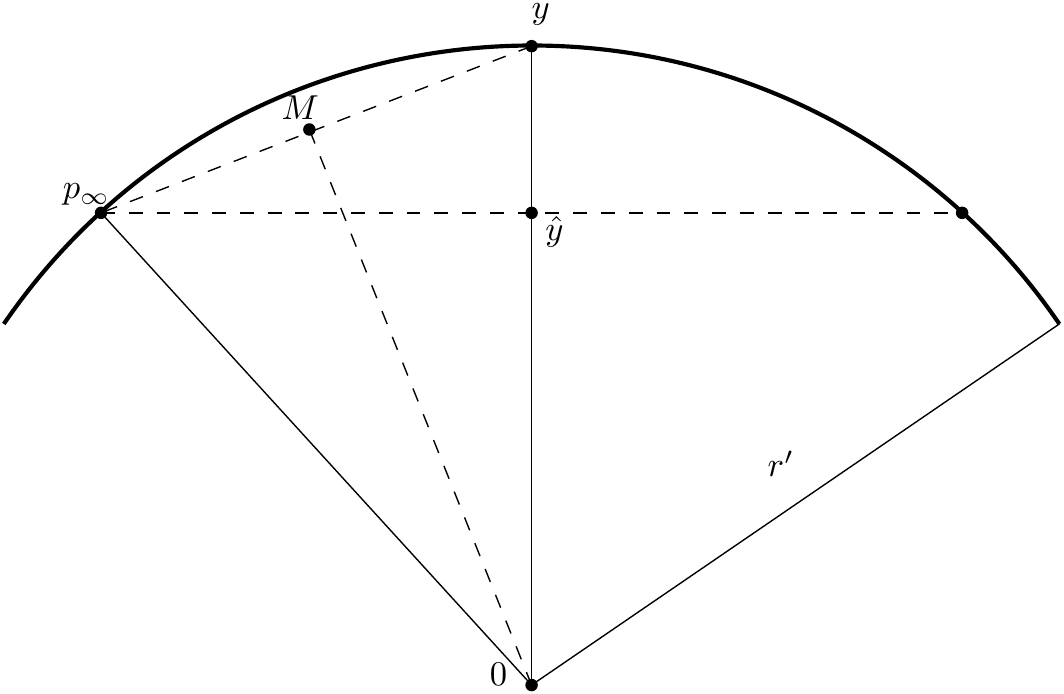}
	    	\caption{A detail from the proof of Lemma \ref{lemmaLatschev}.}
	    \label{figLat2}
	  \end{center}
		\end{figure}

It follows that $\displaystyle \bigcap_{i\in \NN}I_{y,i}=I_{y,\infty}$.  
Let $\hat y$ be the point on the line segment from $0$ to $y$, which is closest to $p_\infty$, see Figure \ref{figLat2} on the left. The radius of the ball centered at $\hat y$ contained in $I_{y,\infty}$ equals $\delta_1=r_\infty - d(\hat y, p_\infty)$. Using elementary geometry on the right side of Figure \ref{figLat2} we see that 
$$
d(y,\hat y)=\frac{r_\infty ^2}{2 r'}, \ \textrm{ and } 
\left(d(p_\infty, \hat y)\right)^2= r_\infty ^2 - \frac{r_\infty ^4}{4 r'^2},
$$
and that the radius $\delta_1$ of the inscribed ball in $I_{y,\infty}$ with center at $\hat y$ (see Figure \ref{figLat2} on the left) equals 
$$
\delta_1(r_\infty, r')=r_\infty - d(p_\infty, \hat y)= r_{\infty} \left( 1-\sqrt{1-\frac{r_\infty^2}{4r'^2}}\right).
$$ 
For a later reference in the proof of Lemma \ref{lemmaLatschev2} we also mention three \emph{structural facts}:
\begin{enumerate}
 \item $\hat y$ lies on the line segment from $y$ to $0$ and thus in $D$.
\item $B(\hat y, \delta_1)$ does not contain $y$. 
 \item    $d(x,\hat y) \leq  r_i-\delta_1$ for any $x\in B(0,r')$ with $d(x,y) \leq r_i$ due to the nesting proved in the claim above.
\end{enumerate}

If $F$ is $\delta_1(r_\infty, r')$ dense there exists $q_y\in F$ contained in the ball mentioned above. Thus $y \mapsto q_y$ induces a $\tilde{r}$-crushing $F \to F\setminus\{y\}$. 

We proceed inductively by treating $F \setminus\{y\}$:
\begin{itemize}
 \item Inductive construction based on Figures \ref{figLat1} and \ref{figLat2} are valid for  $ r' > \frac{r_\infty}{\sqrt{2}}$. During the inductive process ever more indices $i$ fail to satisfy  $r' > \frac{r_i}{\sqrt{2}}$.  However, these indices still satisfy structural fact (3) due to the nesting proved in the claim (the claim does not require this condition) as long as $r_\infty > r'$. As a result we construct a sequence of $\tilde{r}$-crushings to a subset of $F$ contained in $B\left(0,\frac{r_\infty}{\sqrt{2}}\right)\subset B(0,r_\infty)$ and thus $F$ is $\tilde{r}$-crushable. 
 \item  For each $y$ chosen in the induction process $\hat y$ is closer to $0$ and thus hasn't been removed yet, leading to the inductive $\tilde r$-crushings being well defined. 
 \item An important feature of this inductive process is the choice of density $\delta_1$. Note that $\delta_1(r_\infty, r')$ is a decreasing function in $r'$ on $\left[\frac{r_\infty}{\sqrt{2}},\infty\right)$ so the declaration $\delta_1=\delta_1(r_\infty, \alpha)$ suffices for the inductive argument, as $r'$ decreases through the inductive argument. Furthermore, structural fact (3) holds throughout induction for $\delta_1=\delta_1(r_\infty, \alpha)$.
\end{itemize}
\end{proof}

Next we generalize results from $\delta$-dense subsets of $X$ to spaces at proximity $\delta$ to $X$. Let $D\subset B(0,\alpha)\subset \RR^n$ be a star-shaped subset of the standard flat open ball of radius $\alpha>0$. Let $(Y,d)$ be a metric space whose Gromov-Hausdorff distance to $D$ is less than $\delta>0$. As in \cite{Latschev} we define a pseudo-metric $d'$ on disjoint union $Y\sqcup D$ extending $d$ on $Y$ and $d_{eucl}$ on $D$ such that $Y$ is contained in the $\delta$-neighbourhood of $D$ and vice versa, i.e., $D \subseteq N_{d'}(Y, \delta)$ and $Y \subseteq N_{d'}(D, \delta)$. The existence of a pseudo-metric $d'$ follows straight from the definition of the Gromov-Hausdorff metric. Balls in metric $d'$ will be denoted by $B'$.

The following lemma is generalization of Lemma \ref{lemmaLatschev}. The proof is an adaptation of \cite[Lemma 3.1]{Latschev}.

\begin{lemma}\label{lemmaLatschev2}
Let $D\subset B(0,\alpha)\subset \RR^n$ be a star-shaped subset of the standard flat open ball of radius $\alpha$ centered at $0$, and let $\tilde r$ with $\alpha \geq \tilde{r}\geq r_\infty>0$ be fixed. Then  every metric space $(Y,d)$ whose Gromov-Hausdorff distance to $D$ is less than $\delta'_1=\delta_1(r_\infty, \alpha)/\K$ is $\tilde{r}$-crushable.
\end{lemma} 

\begin{proof}
We will be using the notation of the two paragraphs before the statement of Lemma \ref{lemmaLatschev2}.

Fix a finite $\delta'_1$-dense $F \subseteq D$. Note that the $d'$-balls  of radius $2\delta'_1$ around points of $F$ are a cover of $Y \sqcup D$. We proceed as follows:
\begin{itemize}
 \item Let $y\in F$ be at maximal distance from $0$ and let $r'=d_{eucl}(y,0)$.
 \item Choose $\hat y\in D$ as in the proof of Lemma \ref{lemmaLatschev}.
 \item Choose $\hat y'\in Y$ with $d'(\hat y', \hat y) < \delta'_1$.
\end{itemize}

We claim that mapping $Y'= B'(y, 2\delta'_1)\cap Y$ to $\hat y'\in Y$ is an $\tilde r$-crushing on $Y$. Fix $y'\in Y'$ and any positive integer $i$. The claim will be proved if we show that for each $x'\in Y$ with $d'(x',y')< r_i$ we have $d'(x',\hat y')< r_i$. We proceed as follows:
\begin{enumerate}
 \item Choose $x\in F$ with $d'(x',x)<2\delta'_1$ and note that
 $$
 d_{eucl}(y,x)\leq d'(y,y')+d(x',y') + d'(x',x) < r_i + 4 \delta'_1.
 $$
 \item Choose $z\in B(0,r')$ on the line segment from $x$ to $y$ with $d_{eucl}(y,z)< r_i$ and $d_{eucl}(x,z)< 4 \delta'_1$. By the structural fact (3) in the proof of Lemma \ref{lemmaLatschev} we have $d_{eucl}(\hat y,z)< r_i - \delta_1(r_\infty,\alpha)$.
 \item Combining these two observation and $\delta'_1=\delta_1(r_\infty, \alpha)/\K$ we obtain
 $$
 d'(x',\hat y') \leq d'(x',x)+ d_{eucl}(x,z)+ d_{eucl}(z,\hat y) + d'(\hat y, \hat y')< r_i -\delta_1 + \K \delta'_1<r_i.
 $$
\end{enumerate}
Hence the mentioned map is an $\tilde r$-chushing. Note that $d(y, \hat y) \geq \delta_1$ by a construction in the proof of Lemma \ref{lemmaLatschev}. On the other hand, 
$$
d'(y,\hat y') \geq d_{eucl}(y, \hat y) - d'(\hat y, \hat y') \geq 6 \delta' 
 - \delta'_1 = 5 \delta'_1 > 2 \delta'_1 
$$ 
which means $\hat y' \notin Y'$. The mentioned $\tilde r$-crushing thus maps $Y \to Y \setminus Y'$, paving way for the inductive constrution.

We remove $y$ from $F$ and proceed by induction on the obtained finite set $F \setminus \{y\}$ as in Lemma \ref{lemmaLatschev} by iteratively choosing points $y$ at maximal distance from $0$. 
At each inductive step we make a few choices:
\begin{itemize}
 \item We choose $\hat y\in D$ and $\hat y' \in Y$. At this point the point $\hat y'$ should not have been removed from $Y$ yet by the inductive process. This is indeed the case as:
\begin{itemize}
 \item  Each removed point $q'\in Y$ is at distance less than $2 \delta'_1$ from some $q\in F$ with $d_{eucl}(q,0) \geq r'$ and thus satisfies $d'(q',0) \geq r'-2\delta'_1$.
 \item On the other hand, structural fact (2) in the proof of Lemma \ref{lemmaLatschev} implies 
 $$
 d'(\hat y', 0) \leq d_{eucl}(0,\hat y) + d'(\hat y, \hat y') \leq r' - \delta_1 + 2\delta'_1 = r' -5 \delta'_1.
 $$
 \end{itemize}
 \item We choose $x \in F \cap \overline{B(0,r')}$. If the only point of $F$ at distance less than $2 \delta'_1$ from $x'$ was at distance more than $r'$, then $x'$ would already have been removed in the inductive process and wouldn't be involved in the $\tilde r$-crushing at this step.
\end{itemize}

It remains to discuss the termination of the inductive process. Let $0'\in Y$ be a point with $d'(0,0')<\delta'_1$. The process terminates when $r'$ falls below $r_\infty/\sqrt{2}$ and the $d'$-balls of radius $2\delta'_1$ around the remaining points of $F$ form a set contained in $B\left(0,r_\infty/\sqrt{2} + 2\delta'_1\right) \subseteq B'\left(0',r_\infty/\sqrt{2} + 3\delta'_1\right)$. Similarly as in Lemma \ref{lemmaLatschev} we want to prove that this set is contained in $B'\left(0',r_\infty/\sqrt{2} + 3\delta'_1\right)$ as this would imply that the final step can be defined as an $\tilde r$-crushing taking the remaining part of $Y$ into $0'$. In technical terms we need to prove 
\begin{equation}
\label{EqFinal}
r_\infty/\sqrt{2} + 3\delta'_1 \leq r_\infty.  
\end{equation}
Taking into account that (as $\alpha \geq r_\infty$ and )
$$
\delta'_1= \delta_1(r_\infty,\alpha)/7 \leq \delta_1(r_\infty,r_\infty)/7= \frac{r_\infty}{7} \cdot \left( 1-\sqrt{1-\frac{1}{4}}\right)
$$
we can directly verify that
$$
\frac{1}{\sqrt{2}}+ \frac{3}{7} \cdot \left( 1-\sqrt{1-\frac{1}{4}}\right)< 1
$$ 
and thus inequality \ref{EqFinal} holds. 
\end{proof}

\section{Underlying covers}
\label{SubsCovers}

In this section we define convenient covers $\mathcal{C}$ of $X$ and $\mathcal{\overline{C}}$ of $Y$. Fix a closed Riemanninan manifold $X$ and let $\alpha\in (0,  \rho/2)$. Let $A\subset X$ be finite $\frac{\alpha}{2}$-dense in $X$. Let 
$$
\mathcal{C}=\{B(x,\alpha)\ |\ x\in A\}
$$ 
be a finite cover of $X$. By making $\alpha$ smaller we can assume (see Lemma \ref{lemmaMu}) that:
\begin{itemize}
\item the minimal distance between pairs of disjoint balls in $\mathcal {C}$ is positive, and  
\item that the intersection properties of  $\mathcal{C}$ are preserved by small thickenings.
\end{itemize}

\begin{lemma}\label{lemmaMu}
For each finite subset $A \subset X$ there exists a finite set of scales $R_A\subset (0,\infty)$ such that for each $\alpha \in (0,\infty)\setminus R_A$ there exists $\mu>0$ satisfying the following condition for each $\sigma \subseteq A$: $\displaystyle\bigcap_{x\in \sigma} B(x,\alpha)=\emptyset$ iff $\displaystyle\bigcap_{x\in \sigma} B(x,\alpha+\mu)=\emptyset$.
\end{lemma}

\begin{proof}
For each $\sigma\subseteq A$ we define 
$$p_\sigma=\arginf_{r>0}\left\{\bigcap_{x\in \sigma}B(x,r)\neq\emptyset\right \}.$$
Define a finite set $R_A=\{p_\sigma\}_{\sigma\subseteq A}$. Let $p'$ be the smallest number from $R_A$ larger than $\alpha$ (if such a number does not exists we can define $\mu>0$ as any number). Any $\mu<(p'-\alpha)$ satisfies the conclusion of the lemma: if $\displaystyle\bigcap_{x\in \sigma'} B(x,\alpha)=\emptyset$ and $\displaystyle\bigcap_{x\in \sigma'} B(x,\alpha+\mu) \neq \emptyset$ was true for some $\sigma'$, then $[\alpha, \alpha + \mu]$ would contain $p_{\sigma'}$, which is a contradiction with the definition of $\mu$ and $\alpha$.
\end{proof}

Next, let $(Y,d)$ be a metric space whose Gromov-Hausdorff distance to $X$ is less than $\delta>0$. As above we define a pseudo-metric $d'$ on $Y\sqcup X$ as an extension of metrics on individual spaces. If $\delta<\frac{\alpha}{2}$ then from construction $A$ follows that  $d'$-balls $B'$ of radius $\alpha$ and center in $A$ cover $Y\sqcup X$. Note that for each $x\in A$, $B(x,\alpha)=B'(x,\alpha)\cap X$. We denote 
$$\mathcal{\overline{C}}=\{ B'(x,\alpha)\cap Y\ |\ x\in A\}.$$

The following lemma explains that by making $\alpha$ smaller, the corresponding intersections induced by covers $\mathcal{C}$ and  
$\mathcal{\overline{C}}$ are close in the Gromov-Hausdorff distance if $X$ and $Y$ are sufficiently close.

\begin{lemma}\label{LatschevTech}
Let $A \subset X$ be an $s$-dense subset for some $s>0$ and chose  $\delta_2 > 0$. Then $\forall \alpha' > s \ \exists \ell \in (0,\alpha'-s)$ such that $\forall \alpha \in (\alpha'-\ell, \alpha') \ \exists \delta'_2=\delta'_2(\delta_2, \alpha)$ such that the following implication holds:
$$
d_{G,H}(X,Y) < \delta'_2 \implies \forall \sigma \subseteq A: 
d_H \left(\bigcap_{z\in \sigma} B(z,\alpha),\bigcap_{z\in \sigma} B'(z,\alpha)\cap Y \right) < \delta_2.
$$
In particular, for each $\sigma \subseteq A$ we have 
$\displaystyle\bigcap_{z\in \sigma} B(z,\alpha)=\emptyset$ iff $\displaystyle\bigcap_{z\in \sigma} B'(z,\alpha)\cap Y=\emptyset$.
\end{lemma}

\begin{proof}
 Choose a finite $F \subset X$ such that for each $\sigma \subseteq A$ the set 
 $$
 F_\sigma=F \cap \bigcap_{z\in \sigma} B(z,\alpha')
 $$
 is $\delta_2/2$-dense in $\displaystyle\bigcap_{z\in \sigma} B(z,\alpha')$.
 For each $\sigma \subseteq A$ define 
 $\ell_\alpha = \alpha' - \max\{ d(z,y) \mid z\in \sigma, y\in F_\sigma\}$ and let $\ell = \min_{\sigma \subseteq A} \ell_\sigma>0$ denote the leverage. Choose $\alpha \in (\alpha'-\ell, \alpha')$.
For each $\sigma \subseteq A$ define 
 $$
H_\alpha= \bigcap_{z\in \sigma} B(z,\alpha) 
\quad \textit{ and } \quad
H'_\alpha= \bigcap_{z\in \sigma} B'(z,\alpha).
 $$
 Observe that $F_\sigma = F \cap H_\sigma$ by the definition of $\ell$ and that $F_\sigma$ is $\delta_2/2$-dense in $H_\sigma$, hence $d_{H}(F_\sigma, H_\sigma) <\delta_2/2$. Choose $\delta'_2 < \min\{\delta_2/2, \alpha-(\alpha'-\ell), \alpha'-\alpha\}$. We are now in a position to prove $d_H(H_\sigma, H'_\sigma)< \delta_2$. Let us fix $\sigma \subseteq A$:
 
\begin{description}
 \item[Case 1] Let $x\in H_\sigma$. We want to prove there exists $p'_x\in H'_\sigma$ with $d'(x,p'_x)< \delta_2$. We know there exist:
	\begin{itemize}
 		\item $p_x\in F_\sigma : d(x,p_x) < \delta_2/2$;
		\item $p'_x\in Y : d'(p_x,p'_x)< \delta'_2<\delta_2/2$.
	\end{itemize}
	Thus $d'(x,p'_x)< \delta_2$ and it remains to prove that $p'_x\in H'_\sigma$. For each $z \in \sigma$ we have
	$$
	d'(z,p'_x) \leq d'(z,p_x) + d'(p_x,p'_x) \leq (\alpha'-\ell) + (\alpha-(\alpha'-\ell))=\alpha
	$$
	and thus $p'_x\in H'_\sigma$.
\item [Case 2] Let $y\in H'_\sigma$. We want to prove there exists $p_y\in H_\sigma$ with $d'(y,p_y)< \delta_2$.  We know there exists $x_y\in X: d'(x_y,y)< \delta'_2$ and containment $\displaystyle x_y \in \bigcap_{z\in \sigma} B'(z,\alpha')$ follows from the following inequality, which holds for each $z\in \sigma$:
$$
d'(z,x_y) \leq d'(z,y)+d'(y,x_y) < \alpha + \delta'_2
 < \alpha + (\alpha' - \alpha)=\alpha'.
 $$
 There also exists $p_y\in F_\sigma \subset H_\sigma: d(x_y,p_y)< \delta_2/2$.
The following inequality completes the proof:
$$
d'(y,p_y) \leq d'(y,x_y) + d'(x_y,p_y) < \delta'_2 + \delta_2/2 < \delta_2/2 + \delta_2/2 = \delta_2.
$$
\end{description}
\end{proof}

\subsection{Choosing the scale}
\label{SubsScale}

Similarly as in \cite{Latschev} we could scale our chosen $X$ by a large constant so that for each $x\in X$, any subset of the ball $B(x,1)$ was at Gromov-Hausdorff distance less than $\frac{1}{2}\delta'_1\Big((4M)^{-1}, 1\Big)$ from the subset of the flat Euclidean ball of radius $1$ corresponding to it under the exponential map at $x$. Instead we refrain from scaling and choose $\alpha>0$ so small that for each $x\in X$, any subset $B(x,\alpha)$ is at Gromov-Hausdorff distance less than  $\alpha \cdot \frac{1}{2}\delta'_1\Big((4M)^{-1}, 1\Big)$ from the subset of the flat Euclidean ball of radius $\alpha$ corresponding to it under the exponential map at $x$. We observe that
$$
\frac{1}{2}\alpha \cdot \delta'_1\left((4M)^{-1}, 1\right)=
\frac{1}{2}\cdot \frac{1}{\K}\cdot \frac{\alpha}{4M} \left( 1-\sqrt{1-\frac{\left(\frac{\alpha}{4M}\right)^2}{4\alpha^2}}\right)=
\frac{1}{2}\delta'_1\left(\frac{\alpha}{4M}, \alpha\right)
$$
and fix $\delta'_1=\delta'_1\left(\frac{\alpha}{4M}, \alpha\right)$.

\section{The main argument}

We are now in position to prove our main result, Theorem \ref{Latschev}. The strategy of the proof is to define cover
$$
\mathcal{W}=\{\sRips (\overline{C}; \tilde{r})\ |\ \overline{C}\in \mathcal{\overline{C}}\}
$$
(see (9) in the proof below) and establish the following relationship:

$$X\stackrel{(i)}{\simeq} \Nerve(\mathcal{C})\stackrel{(ii)}{\cong} \Nerve(\mathcal {\overline{C}})\stackrel{(iii)}{\cong} \Nerve(\mathcal {W})\stackrel{(iv)}{\simeq} \sRips (Y; \tilde{r} )$$

\begin{proof}[Proof of Theorem \ref{Latschev}]
Let $X$ be a closed Riemannian manifold. Recall $\rho>0$.
\begin{enumerate}
\item Choose scale $\alpha \in (0, \rho/2)$ and the corresponding $\delta'_1$ according to Subsection \ref{SubsScale}.
\item Choose a finite $\frac{\alpha}{2}$-dense subset $A\subset X$.
 \item Use Lemma \ref{lemmaMu} and Lemma \ref{LatschevTech} to potentially decrease $\alpha$ (and $\delta'_1$) so that:
\begin{enumerate}
\item $A$ is still $\frac{\alpha}{2}$-dense in $X$,
 \item by Lemma \ref{lemmaMu} there exists $\mu>0$ satisfying the following condition for each  $\sigma \subseteq A$:
 $$
 \displaystyle\bigcap_{x\in \sigma} B(x,\alpha)=\emptyset
\quad  \textit{    iff    }  \quad
 \displaystyle\bigcap_{x\in \sigma} B(x,\alpha+\mu)=\emptyset.
 $$
 \item by Lemma \ref{LatschevTech} there exists $\delta'_2=\delta'_2\left(\frac{\delta'_1}{2}, \alpha \right )$ such that 
 $$
d_{G,H}(X,Y) < \delta'_2 \implies \forall \sigma \subseteq A: 
d_H \left(\bigcap_{z\in \sigma} B(z,\alpha),\bigcap_{z\in \sigma} B'(z,\alpha)\cap Y \right) < \frac{\delta'_1}{2}.
$$
 \end{enumerate}

\item Cover $$\displaystyle\mathcal {C}=\{ B(x,\alpha)\ |\ x\in A\}$$  is a good open cover of $X$ (see \cite[Lemma 3.5]{LemezVirk}) as $\alpha < \rho/2$.
\item The Nerve Theorem (Theorem \ref{ThmNerve}) implies $X \simeq \Nerve(\mathcal{C}) $ establishing \textbf{(i)}.
\vskip.2cm

\item Set $\delta=\min\left\{\delta'_2,\frac{\alpha}{8}, \frac{\mu}{2} \right\}$.
\item Chose a metric space $Y$ with Gromov-Hausdorff distance to $X$ less than $\delta$.
\item Let $\displaystyle\mathcal{C'}=\{ B'(x,\alpha)\ |\ x\in A\}$ be an open cover of $Y\sqcup X$ as mentioned in Section \ref{SubsCovers}.
\item Define 
$$
\displaystyle\mathcal{\overline{C}}=\{ B'(x,\alpha)\cap Y\ |\ x\in A\}.
$$
\item\label{trd11} Below we prove \label{trd1}$\Nerve(\mathcal {C})\cong \Nerve(\mathcal {\overline{C}})$  establishing \textbf{(ii)}.
\vskip.2cm

\item Define $\varepsilon_0= \alpha/4$ and let $\varepsilon_0 > \tilde r> \varepsilon_0 /M$.
\item Let $\mathcal{W}=\{\sRips (\overline{C}; \tilde{r})\ |\ \overline{C}\in \mathcal{\overline{C}}\}$.
\item It is easy to see that the simplicial map $\varphi_2: \Nerve(\mathcal {\overline{C}})\to \Nerve(\mathcal {W})$ defined by  
$$
B'(x_i,\alpha)\cap Y\mapsto \sRips(B'(x_i,\alpha)\cap Y,\tilde{r})
$$
 is an isomorphism (see \cite[Proposition 3.10]{LemezVirk} for an argument), 
establishing \textbf{(iii)}.
\vskip.2cm

\item It remain to show that $\mathcal {W}$ is a good cover of $\sRips(Y;\tilde{r})$ as small thickenings to a good open cover with isomorphic nerve (for more details see \cite[Remark 3.8]{LemezVirk} and subsequent results) would allow us to use the Nerve Theorem to conclude \textbf{(iv)} and thus complete the first part of the proof. Let us prove that $\mathcal {W}$ is a good cover of $\sRips(Y;\tilde{r})$:
\begin{enumerate}
 \item As $A$ is $\alpha/2$-dense, cover $\mathcal{C}$ is of Lebesgue number at least $\alpha/2$. As $\delta< \frac{\alpha}{8}$, cover $\mathcal{C'}$ is of Lebesgue number at least $\alpha/4=\varepsilon_0 > \tilde{r}$. Thus each simplex of $\sRips(Y;\tilde{r})$ is contained in an element of $\overline{\mathcal{C}}$ and hence $\mathcal {W}$ is a cover of $\sRips(Y;\tilde{r})$.
 \item Assume 
 $$
 L=\sRips\Big(B'(x_1,\alpha)\cap Y; \tilde r\Big) \ \cap \ \ldots \ \cap \ \sRips\Big(B'(x_q,\alpha)\cap Y; \tilde r\Big)= 
 $$
 $$
 =\sRips\Big(B'(x_1,\alpha)\cap \ldots \cap B'(x_q,\alpha) \cap Y; \tilde r\Big)
 $$
 is non-empty for some $x_i\in A$. We proceed as follows:
	\begin{itemize}
	 \item As $\delta < \delta'_2 (\delta'_1/2)$, Lemma \ref{LatschevTech} implies that 
	 $$
	 S_1=B'(x_1,\alpha)\cap \ldots \cap B'(x_q,\alpha) \cap Y
	 $$
	  is at Gromov-Hausdorff distance at most $\delta'_1/2$ from 
	  $$
	 S_2=B(x_1,\alpha)\cap \ldots \cap B(x_q,\alpha).
	 $$
	 \item Condition described in subsection \ref{SubsScale} implies that $S_2$ is at Gromov-Hausdorff distance less than $\delta'_1/2$ from the subset $S_3$ of the flat Euclidean ball of radius $\alpha$ corresponding to it under the exponential map based at any point of $S$. 
	 \item As $S_2$ is convex, $S_3$ is star-shaped by the definition of the exponential map and thus $\tilde{r}$-crushable by Lemma \ref{lemmaLatschev}.
	 \item Now $S_1$ is at Gromov-Hausdorff distance less than $\delta'_1$ from $S_3$ and thus by Lemma \ref{lemmaLatschev2} it is $\tilde{r}$-crushable.
	 \item Complex $L$ is contractible by Corollary \ref{CorollaryCrushable} and $\mathcal {W}$ is thus a good cover.
	\end{itemize}
\end{enumerate}
\end{enumerate}

We conclude the first part by providing the omitted argument of (\ref{trd11}). Let $\varphi_1: \Nerve(\mathcal {C})\cong \Nerve(\mathcal {\overline{C}})$ be a simplicial map, defined by:
$$B(x,\alpha)\mapsto B'(x,\alpha)\cap Y.$$
Let $\sigma=\left\{B(x_0,\alpha),\ldots,B(x_q,\alpha)\right \}$ be a simplex in $\Nerve(\mathcal{C})$. We map such a simplex to $\left\{B'(x_1,\alpha)\cap Y,\ldots,B'(x_q,\alpha)\cap Y\right \}$ in $\Nerve(\mathcal{{\overline{C}}})$. Map $\varphi_1$ is:
\begin{itemize}
 \item bijective on vertices by construction;
 \item  well defined by Lemma \ref{LatschevTech} as $\delta < \delta'_2$;
 \item surjective on simplices, which follows from Lemma \ref{lemmaMu}: if there exists  $$
 z' \in B'(x_1,\alpha)\cap \ldots \cap B'(x_q,\alpha)\cap Y
 $$
then there exists $z\in X$ with $ d'(z,z')< \delta < \mu/2$ and thus 
 $$
 z \in B(x_1,\alpha+ \mu)\cap \ldots \cap B(x_q,\alpha+\mu),
 $$
 which implies 
 $$
 z \in B(x_1,\alpha)\cap \ldots \cap B(x_q,\alpha),
 $$
 by Lemma \ref{lemmaMu}.
\end{itemize}
Thus $\varphi_1$ is an isomorphism.
 
 The second part concerning functoriality follows from the functorial nerve theorem (Theorem \ref{ThmNerve}) for the smaller sequence of scales generates covers $\mathcal{W}_{\tilde{r}}$ and $\mathcal{W}_{\tilde{s}}$. The following diagram commutes up to homotopy:
 $$
 \xymatrix{
 X \ar[r]^{\simeq \qquad }&  \Nerve(\mathcal{C}) \ar@{<->}[r]^{\cong}&  \Nerve(\widehat{\mathcal{C}}) \ar@{<->}[r]^{\cong}& \Nerve(\mathcal{W}_{\tilde{r}}) &  \sRips\left (Y; \tilde r\right)  \ar[l]_{\simeq  }\\
 X \ar[r]^{\simeq \qquad } \ar[u]_{\textrm{id}}& 
  \Nerve(\mathcal{C})  \ar@{<->}[r]^{\cong} \ar[u]_{\textrm{id}}& 
  \Nerve(\widehat{\mathcal{C}})  \ar@{<->}[r]^{\cong} \ar[u]_{\textrm{id}}& \Nerve(\mathcal{W}_{\tilde{s}})  \ar[u]&  
  \sRips\left (Y; \tilde s\right) \ar@{^(->}[u] \ar[l]_{\simeq  }
}
 $$
\end{proof}

\bibliographystyle{model1-num-names}

\end{document}